\documentclass{amsart}

\usepackage{graphicx} 
\usepackage{amsmath,amssymb,amsfonts}
\usepackage{url} 
\usepackage[title]{appendix}
\usepackage{algorithm}
\usepackage{algorithmic}
\usepackage{enumerate}
\usepackage{pgfplots}
\pgfplotsset{compat=1.8}

\theoremstyle{plain}
\newtheorem{theorem}{Theorem}

\newtheorem{lemma}[theorem]{Lemma}

\theoremstyle{definition}

\newtheorem{example}[theorem]{Example}

\theoremstyle{remark}
\newtheorem{remark}[theorem]{Remark}

\newcommand{\F}{\mathcal{F}}
\newcommand{\X}{\mathcal{X}}
\newcommand{\Y}{\mathcal{Y}}
\newcommand{\N}{\mathcal{N}}
\newcommand{\Lc}{\mathcal{L}}
\newcommand{\R}{\mathbb{R}}
\newcommand{\Z}{\mathbb{Z}}

\begin{document}

\title[Separation, CQ and Cycling in OA]{Separation, Constraint Qualifications, and Cycling in Outer Approximation}

\author[E. Tamm]{Erik Tamm}
\address{Department of Mathematics, KTH Royal Institute of Technology, Lindstedtsv\"agen~25, 10044 Stockholm, Sweden}
\email{etamm@kth.se}
\urladdr{https://orcid.org/0009-0008-2878-1458}

\author[J. Kronqvist]{Jan Kronqvist}
\address{Department of Mathematics, KTH Royal Institute of Technology, Lindstedtsv\"agen~25, 10044 Stockholm, Sweden}
\email{jankr@kth.se}
\urladdr{https://orcid.org/0000-0003-0299-5745}

\date{June 22, 2026}

\thanks{This work was funded by the Swedish Research Council (grant number 2022-03502).}


\begin{abstract}
The outer approximation algorithm is a widely used method for solving convex mixed-integer nonlinear programs. While the algorithm is well established in theory and practice, certain assumptions underlying its convergence are rarely discussed in the literature. In this paper, we examine two such assumptions: that a constraint qualification holds at the optimal solution of each nonlinear programming subproblem, and that these subproblems can be solved exactly. We argue that both assumptions are connected to the issue of cycling, by which we mean that the same integer assignment reappears in successive iterations of the algorithm. When a constraint qualification fails, separation of the current iterate from the feasible set is not guaranteed, which can cause the algorithm to stall. When the nonlinear programming subproblem is solved only approximately, separation may likewise fail, and we show that high precision can be required in certain cases. To formalize the connection between these issues, we prove that when Slater's condition is nearly violated, a point close to the exact solution can be found at which separation fails. Furthermore, within the outer approximation algorithm, we propose to use extended cutting planes as a fallback strategy when cycling is detected. We prove that this approach yields a finitely convergent algorithm under relaxed assumptions regarding constraint qualifications, thereby generalizing the convergence theory of the outer approximation algorithm.
\end{abstract}

\maketitle

\section{Introduction}\label{sec:introduction}
The field of mixed-integer nonlinear programming (MINLP) is broad, covering a wide range of problems with numerous applications across science and engineering thanks to the generality of the modelling framework. For an overview of the field, see \cite{belotti2013mixed, kronqvist2025}. An important subclass is so-called convex MINLPs. A problem is commonly considered convex if both the objective function and the linear relaxation of the feasible set are convex. This structural assumption opens the possibility of developing efficient algorithms. One such successful case is the outer approximation (OA) algorithm, first presented in \cite{duran_1986} and later revised in \cite{fletcher_1994}. The underlying idea is that the original problem can be solved by iteratively refining a linear relaxation, that is, an \emph{outer approximation}, of the feasible set through the addition of so-called gradient cuts.

Two closely related algorithms are the extended cutting plane (ECP) algorithm \cite{westerlund_1995} and the extended supporting hyperplane (ESH) algorithm \cite{kronqvist2016extended}. These rely on a similar linear outer approximation idea, but they construct the gradient cuts in a different way compared to the OA algorithm. There are also extensions of the OA algorithm, such as the branch-and-cut method LP/NLP-BB \cite{quesada1992lp}, level-based OA (L-OA) and quadratic OA (Q-OA) \cite{kronqvist2020using}. Furthermore, generalizations have been presented for, e.g., non-smooth problems \cite{eronen_2014, delfino_2018} and some non-convex problems \cite{dambrosio_2009}.

The OA algorithm has been successful in numerical implementations, with many different versions available both as open-source and as commercial software \cite{kronqvist2019review}. These include, for example, DICOPT \cite{dicopt}, BONMIN \cite{bonami2008algorithmic}, MINOTAUR \cite{mahajan2021minotaur}, AIMMS OA \cite{aimms_aoa} and KNITRO \cite{knitro}. The SHOT solver also integrates key components from the OA algorithm \cite{shot}.

In this paper, we discuss and investigate the convergence of the OA algorithm in relation to the assumption of constraint qualifications and to the non-exact solutions of the nonlinear programming (NLP) subproblems. To determine the point at which a gradient cut should be added, the OA algorithm solves an NLP subproblem. For the algorithm to converge, a constraint qualification must hold for this problem. Furthermore, an implicit assumption is that the subproblem can be solved exactly. We investigate the link between these issues, propose a way of handling them, and show that the proposed approach is theoretically sound. This, therefore, generalizes the convergence properties of the OA algorithm under relaxed assumptions regarding constraint qualifications for the NLP subproblems.

The assumption of a constraint qualification in the NLP subproblem is critical to ensure convergence both in theory and in practice. If the solution reported by the NLP subproblem does not satisfy a constraint qualification, the gradient cut that is added is not guaranteed to improve the linear relaxation in a way that prevents the algorithm from stalling. In practice, this means that the same integer assignment may be revisited without making any progress, which we refer to as cycling. However, in practice it is very difficult to verify that a particular constraint qualification holds for all integer assignments, hence some care is needed for handling such situations.

The assumption that the NLP subproblems are solved exactly is implicit, but important to consider in practical implementations. The point obtained from the NLP subproblem is only guaranteed to improve the linear relaxation of the feasible set if the exact solution is found. This is clearly unreasonable in numerical implementations. We will show that arbitrarily high precision of the solution may be required in some circumstances, demonstrating that numerical tolerances should not be ignored.

The issues of constraint qualifications and non-exact solutions are not commonly discussed in the existing literature, but they can lead to convergence failure. The assumption of a constraint qualification is usually included without much discussion, and the exactness of the solution is rarely mentioned. In \cite{kronqvist2020using}, the authors note that some problems in MINLPLib \cite{minlplib} fail to converge with the OA algorithm. Failure to converge can also be observed in \cite[Fig.~3]{kronqvist2019review}, where most solvers fail to converge on certain problems. Note that these failures are not due to time or iteration limits. Furthermore, the reason for such errors is not always known, and for some solvers the issues described here could be reported instead as time-outs.

Some suggestions have been made for handling the issue of cycling. The most naive approach is to use no-good cuts, which remove precisely the integer assignment in question. This is mentioned, for example, in the documentation of DICOPT, where such cuts are referred to as integer cuts \cite{dicopt}. This may not be a computationally efficient way of handling cycling. For general integer variables, no-good cuts require the introduction of binary variables, which increases the size of the problem. They also do not improve the linear relaxation, and they can complicate the computation of lower bounds for the optimal objective value. Another idea, mentioned as an ad-hoc method in \cite{kronqvist2020using}, is to add an ECP cut instead of an OA cut as a fallback strategy. In this paper, we formalize this technique and prove that it is theoretically sound.

The paper is organized as follows. In Section \ref{sec:background}, the convex MINLP problem is formally introduced together with the running example that will be used for illustration purposes. In Section \ref{sec:cycling}, the issue of cycling is introduced and discussed. In Section \ref{sec:CQ}, the importance of constraint qualification assumptions for the OA algorithm is discussed and connected to cycling. In Section \ref{sec:non-exact solutions}, the precision required of the solutions of the NLP subproblems in the OA algorithm is discussed and connected to the two former issues. Two theoretical results are then developed in Section \ref{sec:theory}. First, we establish formally the issue of nearly violating Slater's condition. Second, we show that finite convergence is guaranteed for relaxed assumptions of the OA algorithm with ECP cuts used as a fallback. Specifically, no constraint qualification is assumed for NLP subproblems corresponding to suboptimal integer assignments. Finally, in Section \ref{sec:conclusion}, we summarize our findings.

\section{Background}\label{sec:background}
We consider MINLP problems of the form
\begin{equation}\label{eq:General MINLP}\tag{$P$}
    \begin{array}{rl}
        \min & f(x, y) = c^\top x + d^\top y \\
        \text{s.t.} & g(x, y) \leq 0, \\
                    & x \in \X, y \in \Y.
    \end{array}
\end{equation}
The set $\X \subset \R^n$ is assumed to be a nonempty compact convex set defined by linear inequalities, and the set $\Y \subset \Z^m$ is assumed to be finite. The function $g \colon \R^{n + m} \rightarrow \R^p$ is assumed to be convex and once continuously differentiable. We use the notation $\N = \{(x, y) \in \R^n \times \Z^m \colon g(x, y) \leq 0\}$, $\Lc = \X \times \Y$ and $\F = \N \cap \Lc$, and we denote the optimal objective value by $f^*$. If the integer variables are fixed, we obtain the NLP subproblem
\begin{equation}\tag{NLP($\hat{y}$)}\label{eq:NLP subproblem}
    \begin{array}{rl}
        \min & c^\top x + d^\top \hat{y} \\
        \text{s.t.} & g(x, \hat{y}) \leq 0, \\
                    & x \in \X.
    \end{array}
\end{equation}

The OA and ECP algorithms rely on a common idea. We begin with some initial linear relaxation of the feasible set, giving the so-called mixed-integer linear programming (MILP) subproblem. Solving this subproblem yields a point $(\hat{x}, \hat{y})$. If this point is not feasible for \eqref{eq:General MINLP}, it should be separated from $\F$. To achieve this, a gradient cut is used, which for a constraint $g_i$ and a point $\hat{z} = (\hat{x}, \hat{y})$ is defined as $$g_i(\hat{z}) + \nabla g_i(\hat{z})^\top (z - \hat{z}) \leq 0.$$ We use the terms OA cut and ECP cut to denote the cuts derived for the two algorithms. Both are gradient cuts, but they are derived from different points. The ECP cut is simply a gradient cut at $(\hat{x}, \hat{y})$. To derive the OA cut, we solve \eqref{eq:NLP subproblem} to obtain an optimal solution $\bar{x}$. The OA cut is then simply a gradient cut at $(\bar{x}, \hat{y})$. If \eqref{eq:NLP subproblem} is infeasible, it is handled by means of the so-called NLP feasibility problem in order to find a point for the gradient cut. Details of the OA algorithm can be found in \cite{fletcher_1994}, and of the ECP algorithm in \cite{westerlund_1995}. A pseudocode summary of the algorithms has also been included in Appendices \ref{app:OA} and \ref{app:ECP} for reference.

\subsection{Running example}\label{subsec:running example}
We will use a simple MINLP as a running example for illustration purposes. Consider, for $r \in [1, \sqrt{2})$, the problem
\begin{equation}\label{eq:running example}
    \begin{array}{rl}
        \min & -2x - y \\
        \text{s.t.} & x^2 + y^2 \leq r^2, \\
                    & x \in [-1, 1], \\
                    & y \in \{0, 1\}. \\
    \end{array}
\end{equation}
We use $g(x, y) = x^2 + y^2 - r^2$ to denote the nonlinear constraint, so that the constraint reads $g(x, y) \leq 0$. Let $\F$ denote the feasible set.

If the nonlinear constraint is ignored, the optimal solution is $(1, 1)$. We will therefore consider the case in which $(1, 1)$ is to be separated from $\F$. In the OA algorithm, \eqref{eq:running example} is solved with $y$ fixed to $1$ in order to find a corresponding $x$, and a gradient cut is added at the point found. Fixing $y$ to $1$ yields
\begin{equation*}
    \begin{array}{rl}
        \min & - 2x - 1 \\
        \text{s.t.} & x^2 \leq r^2 - 1, \\
                    & x \in [-1, 1].
    \end{array}
\end{equation*}
The solution is then clearly $x = \sqrt{r^2 - 1}$. Hence $g$ is linearized at $(\sqrt{r^2 - 1}, 1)$ to obtain the gradient cut. In the ECP algorithm, a gradient cut is derived at the point $(1, 1)$, so the cut becomes $x + y \leq 1 + r^2 / 2$. We will return to this example throughout Sections \ref{sec:cycling}--\ref{sec:non-exact solutions}.

\section{Cycling}\label{sec:cycling}
Cycling, or stalling, is an issue for many optimization algorithms, including the OA algorithm. This section discusses a few reasons why it can occur for the OA algorithm.

Cycling, in the sense of this paper, occurs when the same solution is found in successive iterations. We wish to distinguish between two different situations. First, separation of the solution of the MILP problem can fail. When the OA algorithm functions normally, the point found by the MILP subproblem is separated from the feasible set by the added cut. If separation fails, the same point as before is again optimal for the MILP problem, so the algorithm stalls and no progress is made. Second, separation of the specific MILP solution may be achieved, but the same integer assignment can still remain feasible. In this case, although the previous MILP solution is no longer feasible, the situation might not cause any problem for the algorithm. If the distance from the cut to the MILP solution is small, it can nonetheless lead to stalling or convergence failure, since the improvement of the linear approximation of the feasible set is not sufficient. In theory, this should not happen to the OA algorithm, but for reasons discussed later it can still occur in numerical implementations.

Since we wish to consider both of these situations, we use the term cycling rather than failure of separation. While both situations are linked to the question of separation, the second case is not necessarily a failure of separation, but rather a matter of varying degrees of separation strength. We illustrate the issue of cycling with Example \ref{ex:cycling}.

As noted in the introduction, no-good cuts are one alternative for handling cycling, but they enlarge the problem and do not improve the linear relaxation. In Subsection~\ref{sec:finite convergence}, we argue that adding ECP cuts to the OA algorithm is a better way of handling the issue.

\begin{example}\label{ex:cycling}
Consider the running example \eqref{eq:running example} with $r = 1$, illustrated in Figure \ref{fig:cycling}. Recall that we are using an OA or ECP cut to separate $(1, 1)$ from $\F$. The OA cut is added at $(0, 1)$, which yields the cut $y \leq 1$. We see that $(1, 1)$ is not separated from $\F$, so the OA algorithm would cycle, since $(1, 1)$ remains the optimal solution of the MILP subproblem. For an ECP cut, we obtain the inequality $x + y \leq 1.5$, which does separate $(1, 1)$ from $\F$. However, $(0.5, 1)$ is an optimal solution of the MILP subproblem once the cut is added, so the integer assignment $\hat{y} = 1$ is not completely separated from $\F$, even though the approximation of $\F$ has improved.

\begin{figure}[ht]
    \centering
    \begin{tikzpicture}
        \begin{axis}[width = 0.8\textwidth,
                    xmin = -1.5,
                    xmax = 1.5,
                    ymin = 0,
                    ymax = 1.5,
                    axis equal image,
                    ymajorgrids=true,
                    grid style=dashed,
                    xlabel = $x$,
                    ylabel = $y$,
                    ytick = {0, 1}]
            \draw[black] (axis cs:0,0) circle[radius = 1];
            \addplot[mark = *] coordinates{(1, 1)} node[above] {(1,1)};
            \addplot[mark = *] coordinates{(0, 1)} node[below] {(0,1)};
            \addplot[domain = -1.5:1.5, samples = 10]{1};
            \addplot[domain = -1.5:1.5, samples = 10]{1.5-x};
        \end{axis}
    \end{tikzpicture}
    \caption{Illustration of Example \ref{ex:cycling}. The OA cut is $y \leq 1$ and the ECP cut is $x + y \leq 1.5$.}
    \label{fig:cycling}
\end{figure}
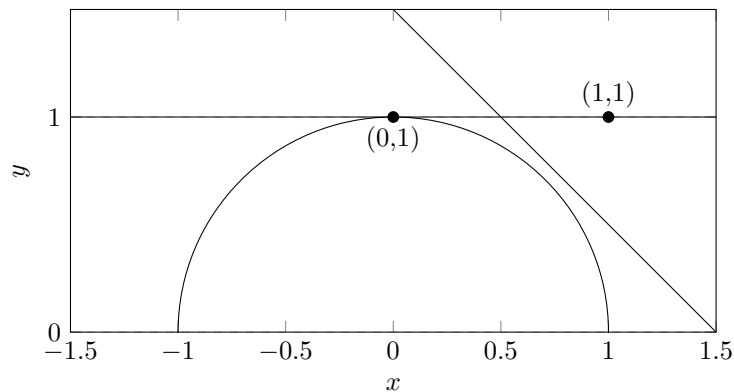
\end{example}

\section{Constraint qualification}\label{sec:CQ}
Constraint qualifications are important in optimization since they exclude degenerate and other edge cases. A classic overview of constraint qualifications by Mangasarian can be found in \cite{mangasarian_1969}. This section describes common constraint qualification assumptions for the OA algorithm, explains why they are critical for convergence, and discusses their connection to separation.

A constraint qualification ensures that the Karush--Kuhn--Tucker (KKT) conditions are necessary for locally optimal solutions, but we believe that a more geometric viewpoint can be helpful. The most general constraint qualification for convex optimization is Abadie's constraint qualification \cite{peterson_1973}. The condition states that the tangent cone to the feasible set at a feasible point is equal to the set of first-order feasible directions at that point. This means that linearizing the constraints at a given point approximates the feasible set well enough locally. Since the idea of algorithms such as OA is to linearize the constraints in order to approximate the feasible set, this assumption is important for ensuring separation. If $(\hat{x}, y^k)$ is obtained by solving the MILP subproblem and $\bar{x}$ is an optimal solution of (NLP($y^k$)), then the direction $\hat{x} - \bar{x}$ is clearly a descent direction for (NLP($y^k$)), since otherwise the MILP solution would not be a minimizer. The linearization of the constraints must therefore be able to restrict that direction. This is exactly what is ensured when the set of first-order feasible directions equals the tangent cone. Constraint qualifications are thus inherently linked to the question of separation.

This argument is used in the proof of the finite convergence of the OA algorithm in \cite[Thm.~2]{fletcher_1994}. The assumption of a constraint qualification implies that no feasible descent direction exists at an optimal solution of the NLP subproblem. The reason follows from the discussion above: the linearized constraints approximate the feasible set well enough.

Let us again consider the running example with $r = 1$ to illustrate a case in which separation fails because a constraint qualification is not satisfied. When $y$ is fixed to $1$, Abadie's condition does not hold at $(0, 1)$, which is the only feasible point for the NLP. Note that this implies that neither Slater's condition nor the LICQ holds. Separation is then not achieved, as was seen in Example \ref{ex:cycling}, and the OA algorithm does not converge. In this case, cycling and the failure to satisfy a sufficient constraint qualification are linked.

In the foundational papers on the OA algorithm, different constraint qualifications are assumed. In the paper by Duran \& Grossmann \cite{duran_1986}, generally considered the first to describe the OA algorithm, Slater's constraint qualification is assumed for the NLP subproblems. In Fletcher \& Leyffer \cite{fletcher_1994}, on the other hand, the assumption is that some constraint qualification holds, with Abadie's constraint qualification given as an example. The authors also note that the LICQ implies Abadie's condition. Slater's condition likewise implies Abadie's condition \cite{peterson_1973}. For the ECP algorithm, no constraint qualification is needed, since no NLP subproblems are solved \cite{westerlund_1995}. In practice, Slater's condition is, arguably, the most practical assumption since it is easier to verify. For more general constraint qualifications, like Abadie's condition, checking whether they are satisfied is often computationally intractable in practical applications.

It is important to note that, since both Slater's condition and the LICQ imply Abadie's condition, they are sufficient but not necessary. Consider the simple case of the two constraints $x^2 + x \leq 0$ and $x^2 - x \leq 0$. Then $x = 0$ is the only feasible point, and at this point neither Slater's condition nor the LICQ holds. Linearizing at $x = 0$ gives the cuts $x \leq 0$ and $-x \leq 0$, so any infeasible point is separated and separation never fails. Note that for both the original constraints and the linearized constraints, $x = 0$ is the only feasible point. Therefore, both the tangent cone and the set of first-order feasible directions contain only $0$, so Abadie's condition holds.

The importance of constraint qualifications is also apparent when one compares the OA algorithm with its generalizations to non-differentiable functions. In \cite{eronen_2014}, the authors observe that subgradients must satisfy the KKT conditions for the NLP subproblem if separation is to be achieved. This again illustrates the central role of constraint qualifications, since they are what ensure that optimal solutions satisfy the KKT conditions.

\section{Non-exact solutions}\label{sec:non-exact solutions}
An important, but often implicit, assumption in optimization algorithms is that the subproblems can be solved exactly. When properties of the OA algorithm are analysed or proved, we implicitly assume that the NLP subproblems can be solved exactly. This is, of course, not true in practice when numerical optimization methods are used. In this section, we describe how non-exact solutions of the NLP subproblem can cause the OA algorithm to fail.

In most cases, small errors in the solution do not give rise to any critical issues for the algorithm. There are, however, situations in which the algorithm fails if a solution is not exact enough. Provided that all assumptions are fulfilled, the gradient cut added in the OA algorithm separates the last solution found, as long as the constraints are linearized at the exact NLP solution. There is, however, no guarantee that separation will succeed for an arbitrary approximate solution. To analyse which points can be used to derive separating gradient cuts, we may for the moment ignore the integrality constraints.

Consider the constraint $x_1^2 + \dots + x_n^2 \leq r^2$, that is, a sphere of radius $r$ in $\R^n$. Suppose that we wish to separate $\bar{x} \in \R^n$ from this set. We can construct a gradient cut at $a \in \R^n$ as
\begin{equation*}
    \sum_{i = 1}^n a_i^2 - r^2 + \sum_{i = 1}^n 2a_i(x_i - a_i) \leq 0.
\end{equation*}
Rearranging the inequality gives
\begin{equation*}
    \sum_{i=1}^n 2a_i x_i - a_i^2 \leq r^2,
\end{equation*}
which can be written as
\begin{equation*}
    \sum_{i = 1}^n x_i^2 - (x_i - a_i)^2 \leq r^2.
\end{equation*}
This cut fails to separate $\bar{x}$ if $\bar{x}$ satisfies this inequality, that is, if
\begin{equation}\label{eq:Unit sphere separation}
    R^2 := \sum_{i = 1}^n \bar{x}_i^2 - r^2 \leq \sum_{i=1}^n (\bar{x}_i - a_i)^2.
\end{equation}
Hence any point $a$ outside the sphere of radius $R$ centred at $\bar{x}$ fails to achieve separation.

We use this result in the running example to illustrate the situation.
\begin{example}\label{ex:non-exact solution}
    We can use the expression \eqref{eq:Unit sphere separation} to investigate at which points the nonlinear constraint can be linearized so as to separate $(1, 1)$ from $\F$. Here $\bar{x} = (1, 1)$, so any point outside the sphere of radius $\sqrt{2 - r^2}$ centred at $(1, 1)$ fails to achieve separation. In the figures, $a$ denotes the point at which the gradient cut is derived, and $x^k$ is the exact solution of the NLP subproblem with $y = 1$.

    We first consider the case $r = 1$, shown in Figure \ref{fig:non-exact_r_1}. The circle within which any linearization gives a separating hyperplane is the unit circle centred at $(1,1)$. This circle passes through $(0, 1)$, which means that any point $(-\varepsilon, 1)$ with $\varepsilon > 0$ fails to produce a separating hyperplane, whereas any point $(\varepsilon, 1)$ with $\varepsilon > 0$ separates $(1, 1)$ from $\F$. In this case the solution need not be exact, but the error must be in the ``correct'' direction in order to achieve separation. Linearizing at the point $a$ gives the plotted gradient cut, which does not separate $(1, 1)$.

    We now consider the case $r = 1.1$. Here the circle centred at $(1, 1)$ does not intersect the exact solution; instead, there is a neighbourhood of the exact solution within which separation is achieved. The point $a$ is the same distance from the exact solution as in the case $r = 1$. Since $a$ now lies inside the circle, the gradient cut does separate $(1, 1)$.

    \begin{figure}[ht]
        \centering
        \begin{tikzpicture}
            \begin{axis}[width = 0.8\textwidth,
                        xmin = -1.5,
                        xmax = 1.5,
                        ymin = 0,
                        ymax = 1.5,
                        axis equal image,
                        ymajorgrids=true,
                        grid style=dashed,
                        xlabel = $x$,
                        ylabel = $y$,
                        ytick = {0, 1}]
                \draw[black] (axis cs:0,0) circle[radius = 1];
                \addplot[mark = *] coordinates{(1, 1)} node[below] {(1,1)};
                \addplot[mark = *] coordinates{(0, 1)} node[below left] {$x^k$};
                \addplot[domain = -1.5:1.5, samples = 10]{1};
                \draw[black] (axis cs:1,1) circle[radius = 1];
                \addplot[mark = *] coordinates{(-0.2, 1)} node[above left] {$a$};
                \addplot[domain = -1.5:1.5, samples = 50]{0.2*x+1.02}; 
            \end{axis}
        \end{tikzpicture}
        \caption{Running example with $r = 1$. The gradient cut at the point $a$ does not separate $(1,1)$ from $\F$.}
        \label{fig:non-exact_r_1}
    \end{figure}
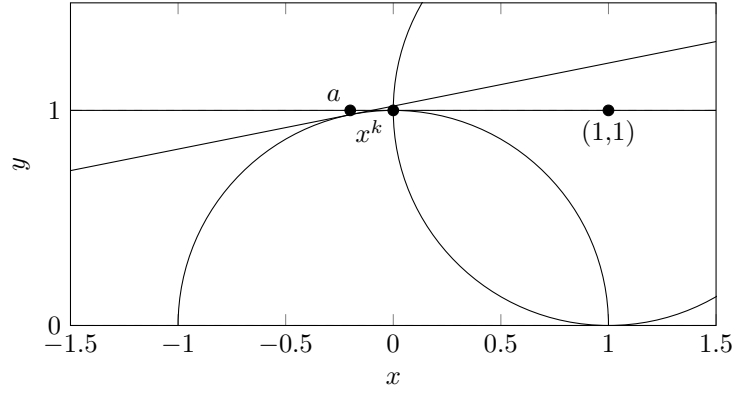

    \begin{figure}[ht]
        \centering
        \begin{tikzpicture}
            \begin{axis}[width = 0.8\textwidth,
                        xmin = -1.5,
                        xmax = 1.5,
                        ymin = 0,
                        ymax = 1.5,
                        axis equal image,
                        ymajorgrids=true,
                        grid style=dashed,
                        xlabel = $x$,
                        ylabel = $y$,
                        ytick = {0, 1}]
                \draw[black] (axis cs:0,0) circle[radius = 1.1];
                \addplot[mark = *] coordinates{(1, 1)} node[above] {(1,1)};
                \addplot[mark = *] coordinates{(sqrt(1.1^2-1), 1)} node[above right] {$x^k$};
                \addplot[domain = -1.5:1.5, samples = 10]{1};
                \draw[black] (axis cs:1,1) circle[radius = 0.79];
                \addplot[mark = *] coordinates{(sqrt(1.1^2-1)-0.2, 1)} node[below right] {$a$};
                \addplot[domain = -1.5:1.5, samples = 50]{1.14-0.26*x};
            \end{axis}
        \end{tikzpicture}
        \caption{Running example with $r = 1.1$. The gradient cut at the point $a$ does separate $(1,1)$ from $\F$.}
        \label{fig:non-exact_r_1.1}
    \end{figure}
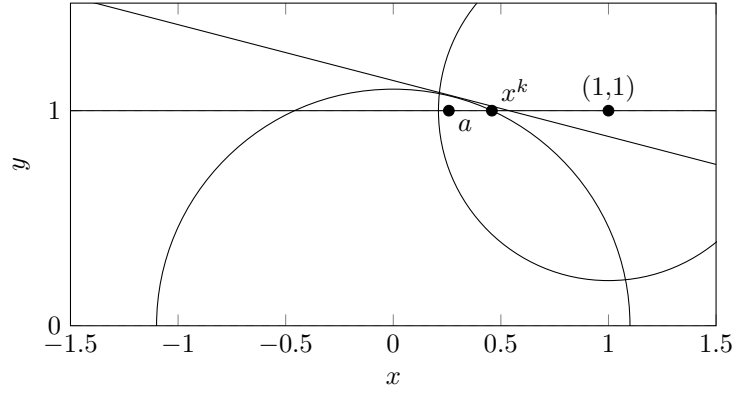
\end{example}

We conclude from Example \ref{ex:non-exact solution} that the accuracy required for separation is not constant. In the first case, a non-exact solution was acceptable in some directions but not in others, whereas in the second case any solution up to some precision was acceptable. Note that neither Slater's condition nor the LICQ holds when $r = 1$, while, for instance, Slater's condition holds when $r = 1.1$. This suggests a connection between constraint qualifications and the accuracy required for separation.

\section{Two theoretical results}\label{sec:theory}
This section develops two results that together address the issues raised above: the first formalizes the connection between constraint qualifications and the precision required of the NLP solution, and the second shows that ECP cuts, added in conjunction with OA cuts, repair the convergence of the OA algorithm under relaxed assumptions.

\subsection{Separation and connection to Slater's condition}\label{subsec:slater}
To formalize the connection suggested by Example \ref{ex:non-exact solution}, our goal is to show that, when Slater's condition is nearly violated, one can construct a point close to the point used for the OA cut at which separation fails. As a consequence, a high precision is needed for the NLP subproblem when Slater's condition is nearly violated. In practical OA algorithms, only active constraints are used to derive gradient cuts. This motivates our result, since we show that we can find a point at which separation fails when only active constraints are used to derive the gradient cuts.

\begin{theorem}\label{thm:Slater}
Consider \eqref{eq:General MINLP}. Let $(\hat{x}, y^k)$ be a solution of a MILP subproblem that is not feasible for \eqref{eq:General MINLP}. Assume that (NLP($y^k$)) is feasible and satisfies Slater's condition, and assume further that $\N$ is a compact set. Then there is a point $(x, y^k) \in \N$ such that adding gradient cuts at this point for the active constraints does not separate $(\hat{x}, y^k)$ from $\F$.
\end{theorem}
\begin{proof}
Since Slater's condition is fulfilled, there exists $\bar{x} \in \R^n$ and $\bar{\alpha} \in \R$ such that $$g_i(\bar{x}, y^k) \leq \bar{\alpha} < 0$$ for each $i$. Consider the points of the form $\tilde{x}(t) = \bar{x} + t(\bar{x} - \hat{x})$. Note that $\bar{x} \neq \hat{x}$. Let $$I(t) := \{i \in \{1, \ldots, p\} \colon g_i(\tilde{x}(t), y^k) = 0\}.$$ Since $\N$ is compact and each $g_i$ is continuous, there exists $\tilde{t} > 0$ such that $$g_i(\tilde{x}(\tilde{t}), y^k) \leq 0$$ for every $i$ and $I(\tilde{t}) \neq \varnothing$, that is, $(\tilde{x}(\tilde{t}), y^k)$ lies on the boundary of $\N$. Adding cuts at $\tilde{x}(\tilde{t})$ means deriving gradient cuts from $g_i$ for $i \in I(\tilde{t})$. By the convexity of $g_i$, we have
\begin{equation*}
    g_i(\tilde{x}(\tilde{t}), y^k) + \nabla g_i(\tilde{x}(\tilde{t}), y^k)^\top \begin{pmatrix}
        \bar{x} - \tilde{x}(\tilde{t}) \\ y^k - y^k
    \end{pmatrix} \leq g_i(\bar{x}, y^k).
\end{equation*}
Since $g_i(\tilde{x}(\tilde{t}), y^k) = 0 > g_i(\bar{x}, y^k)$ for $i \in I(\tilde{t})$, we obtain
\begin{equation*}
    \nabla g_i(\tilde{x}(\tilde{t}), y^k)^\top \begin{pmatrix}
        \bar{x} - \tilde{x}(\tilde{t}) \\ 0
    \end{pmatrix} < 0.
\end{equation*}

Note that $\hat{x} - \tilde{x}(\tilde{t}) = \lambda(\bar{x} - \tilde{x}(\tilde{t}))$ with $\lambda = 1 + 1/\tilde{t} > 0$. Therefore,
\begin{equation*}
    \nabla g_i(\tilde{x}(\tilde{t}), y^k)^\top \begin{pmatrix}
        \hat{x} - \tilde{x}(\tilde{t}) \\ 0
    \end{pmatrix}
    = \lambda \, \nabla g_i(\tilde{x}(\tilde{t}), y^k)^\top \begin{pmatrix}
        \bar{x} - \tilde{x}(\tilde{t}) \\ 0
    \end{pmatrix}
    \leq 0.
\end{equation*}

We conclude that
\begin{equation*}
    g_i(\tilde{x}(\tilde{t}), y^k) +
    \nabla g_i(\tilde{x}(\tilde{t}), y^k)^\top \begin{pmatrix}
        \hat{x} - \tilde{x}(\tilde{t}) \\ 0
    \end{pmatrix}
    \leq 0
\end{equation*}
for all $i \in I(\tilde{t})$. Hence the gradient cut does not separate $(\hat{x}, y^k)$ from $\F$, assuming that only active constraints are linearized.
\end{proof}

\begin{remark}
    The definition of $I(t)$ can be extended to include all indices $i$ for which $g_i(x, y) \geq \bar{\alpha}$ without changing the proof. Consequently, the use of tolerances to determine which constraints are active in a numerical setting does not pose a problem for the theoretical construction of a point where separation fails.
\end{remark}

Our interpretation of Theorem \ref{thm:Slater} is as follows. Consider the situation in which $\N$ is contained in a small open ball. This means that the feasible set of \eqref{eq:NLP subproblem} is contained in some neighbourhood $B_\varepsilon(x^k)$. If $\varepsilon$ is small, then Slater's condition is nearly violated, since the interior of the feasible set is also small. This in turn implies that $\|x^k - \tilde{x}\| < \varepsilon$. Hence the precision of the solution of \eqref{eq:NLP subproblem} must be high, as there is otherwise a risk of failing to separate $(\hat{x}, y^k)$. This idea connects the three concepts of cycling, constraint qualification and non-exact solution of the NLP subproblem.

\subsection{Finite convergence of OA with relaxed assumptions}\label{sec:finite convergence}
Assume that we are solving \eqref{eq:General MINLP} using the OA algorithm from \cite{fletcher_1994}. The subproblems \eqref{eq:NLP subproblem} are now those obtained by fixing $y = \hat{y} \in \Y$ in \eqref{eq:General MINLP}. In \cite{fletcher_1994}, it is assumed that a constraint qualification holds at all optimal solutions of \eqref{eq:NLP subproblem} whenever the problem is feasible. This assumption will be relaxed in this section. We instead assume that a constraint qualification holds at all optimal solutions of \eqref{eq:NLP subproblem} where $\hat{y}$ corresponds to an optimal solution of \eqref{eq:General MINLP}. We will prove that for this relaxed assumption, finite convergence is achieved if the OA algorithm is combined with ECP cuts.

We now adapt the OA algorithm by changing how cuts are added. Let $$\Y^* = \{y \in \Z^m \colon \exists x \in\R^n  \text{ such that } (x, y) \text{ is optimal for \eqref{eq:General MINLP}}\}$$ denote the set of integer assignments corresponding to optimal solutions of \eqref{eq:General MINLP}. Clearly $\Y^* \subset \Y$. When a MILP solution $(\hat{x}, \hat{y})$ is obtained, we distinguish between three cases depending on $\hat{y}$.
\begin{enumerate}[(i)]
    \item If \eqref{eq:NLP subproblem} is infeasible, an OA cut from the NLP feasibility problem is added.
    \item If $\hat{y} \in \Y^*$, an ordinary OA cut is added.
    \item If \eqref{eq:NLP subproblem} is feasible but $\hat{y} \notin \Y^*$, an ECP cut is added.
\end{enumerate}
We will prove that this scheme terminates in a finite number of steps provided that a constraint qualification holds at a solution of every subproblem \eqref{eq:NLP subproblem} with $\hat{y} \in \Y^*$.

The algorithm presented here is mainly a theoretical construction. In practice, ECP cuts would be added only when cycling is detected in the OA algorithm. Our results show that only a finite number of ECP cuts are needed to break this cycling, even though the ECP algorithm converges only in the limit.

The following lemma shows that the adapted algorithm generates an integer assignment from $\Y^*$ in a finite number of iterations, essentially by proving that the ECP algorithm will do the same. This will be an important fact for proving finite convergence of the algorithm.
\begin{lemma}\label{lem:ECP finite to optimal}
    Let $\{(x^k, y^k)\}$ be a sequence generated by the adapted algorithm. Then there exists $K$ such that $y^K \in \Y^*$.
\end{lemma}
\begin{proof}
    If $\hat{y} \in \Y$ but \eqref{eq:NLP subproblem} is infeasible, the OA cut generated will separate $\hat{y}$ from the feasible set \cite[Lemma 1]{fletcher_1994}. Since $\Y$ is finite, a finite number of such cuts is sufficient to separate all of them. We can therefore assume without loss of generality that all $\hat{y} \in \Y$ have a feasible subproblem \eqref{eq:NLP subproblem}.

    Since $\Y$ is finite, we can enumerate $\Y \setminus \Y^*$ as $y_1, y_2, \ldots, y_M$ where $M = \lvert \Y \setminus \Y^* \rvert$. Let $\gamma_i$ be the optimal value of (NLP($y_i$)) and define $\gamma = \min_i \gamma_i$. Then $\gamma$ is the best objective value attained by any suboptimal integer assignment, and it clearly holds that $\gamma > f^*$.

    Assume, for contradiction, that there is no $K$ with $y^K \in \Y^*$. This reduces the algorithm to the ECP algorithm since only ECP cuts are added. Since the sequence $\{(x^k, y^k)\}$ belongs to a compact set, there exists a convergent subsequence $(x^{k_j}, y^{k_j})$ with $y^{k_j} \notin \Y^*$. By the convergence of the ECP algorithm, this subsequence converges to a feasible point $(\bar{x}, \bar{y})$, that is, $g(\bar{x}, \bar{y}) \leq 0$. Hence there exists a $J$ such that $$f(\bar{x}, \bar{y}) - f(x^{k_J}, y^{k_J}) < \gamma - f^*.$$ It holds that $f(\bar{x}, \bar{y}) \geq \gamma$ since $\bar{y} \in \Y \setminus \Y^*$ and $\bar{x}$ is feasible for (NLP($\bar{y}$)). We then obtain
    \begin{equation*}
        \gamma - f^* > f(\bar{x}, \bar{y}) - f(x^{k_J}, y^{k_J}) \geq \gamma - f(x^{k_J}, y^{k_J}).
    \end{equation*}
    Rearranging gives $f(x^{k_J}, y^{k_J}) > f^* = f(x^*, y^*)$ for some optimal solution $(x^*, y^*)$ of \eqref{eq:General MINLP}. Since $(x^*, y^*)$ is feasible, it is feasible for every MILP subproblem, so $(x^*, y^*)$ would be returned instead of $(x^{k_J}, y^{k_J})$. This is a contradiction, which proves the lemma.
\end{proof}

The following lemma is essentially contained in the proof of Theorem~2 in \cite{fletcher_1994}. We include it because the result is used in Theorem \ref{thm:algo finite}.
\begin{lemma}\label{lem:no suboptimal points}
    Assume that an OA cut is added at $(\hat{x}, \hat{y})$ and that a constraint qualification holds at $\hat{x}$ for \eqref{eq:NLP subproblem}. Then for any $x$ such that $(x, \hat{y})$ satisfies the cut inequality, it holds that $f(x, \hat{y}) \geq f(\hat{x}, \hat{y})$.
\end{lemma}
\begin{proof}
    Since $\hat{x}$ is an optimal solution of \eqref{eq:NLP subproblem} and a constraint qualification holds for \eqref{eq:NLP subproblem}, there is no feasible descent direction. That is,
    \begin{equation*}
        \nabla f(\hat{x}, \hat{y})^\top \begin{pmatrix}
            x - \hat{x} \\ 0
        \end{pmatrix}
        \geq 0
    \end{equation*}
    for any feasible $x$. By the convexity of $f$, we get
    \begin{equation*}
        f(x, \hat{y}) \geq f(\hat{x}, \hat{y}) + \nabla f(\hat{x}, \hat{y})^\top \begin{pmatrix}
            x - \hat{x} \\ 0
        \end{pmatrix}
        \geq f(\hat{x}, \hat{y}).
    \end{equation*}
\end{proof}

Using Lemma \ref{lem:ECP finite to optimal}, we can now prove the main result. Theorem \ref{thm:algo finite} shows that the constructed algorithm converges in a finite number of iterations, which implies that adding ECP cuts can resolve the issue of cycling in the OA algorithm using only a finite number of ECP cuts.
\begin{theorem}\label{thm:algo finite}
    The proposed algorithm converges in a finite number of iterations.
\end{theorem}
\begin{proof}
    Since infeasible integer assignments are handled with OA cuts, they are removed in a finite number of steps. Lemma \ref{lem:ECP finite to optimal} states that a finite number of steps is needed to obtain a point $(\bar{x}, y^*)$ with $y^* \in \Y^*$. Since an OA cut is added at this step, a cut is added at $(x^*, y^*)$, that is, at an optimal solution. The algorithm thereby obtains $f^*$ as an upper bound on the optimal objective value. It remains for the algorithm to conclude that the lower bound is also $f^*$.

    We can now view this as restarting the algorithm, since the only difference is the initial linear relaxation. By Lemma \ref{lem:ECP finite to optimal}, only a finite number of steps is again needed to obtain a point $(\bar{x}, y^*)$ with $y^* \in \Y^*$. This process can then be repeated. Since $\Y^*$ is finite, because $\Y$ is finite, the process can repeat only a finite number of times before some $y^*$ is repeated. By Lemma \ref{lem:no suboptimal points}, the objective value of this point is greater than or equal to $f(x^*, y^*)$. Since this is the optimal objective value of the MILP subproblem, the lower bound on the optimal objective value of \eqref{eq:General MINLP} is $f^*$. Hence the upper and lower bounds coincide, which shows that the algorithm converges in a finite number of iterations.
\end{proof}

In practice, of course, we cannot know beforehand whether $\hat{y}$ lies in $\Y^*$. Even if it were possible, implementing this algorithm directly would not be advisable. Instead, the standard OA algorithm should be used, with an ECP cut serving as a fallback. As has been discussed, the OA algorithm ensures in theory, with all assumptions satisfied, that no $\hat{y}$ is repeated. Hence, whenever some $\hat{y}$ has been obtained twice, an ECP cut is added instead of an OA cut. Since Theorem~\ref{thm:algo finite} establishes finite convergence of the constructed algorithm, this fallback scheme is also guaranteed to converge. The fallback scheme may also circumvent issues due to numerical tolerances as long as the subproblems with optimal integer assignments are solved accurately enough. It therefore provides a practical and simple way of handling cycling while remaining theoretically sound.

The convergence result lies between those of the OA algorithm and the ECP algorithm. For the OA algorithm, convergence is achieved in at most $|\Y|$ iterations, while the ECP algorithm does not converge finitely. For our adapted algorithm, convergence is finite but without a bound on the number of iterations. Since the ECP algorithm does not have a bound on the number of iterations, we doubt that it is possible to find such a bound.

\section{Conclusion}\label{sec:conclusion}
In this paper, we have discussed several reasons why the OA algorithm can fail to converge. The main issue is cycling, by which we mean that integer assignments reappear in the OA algorithm; in theory this should not happen. We argued that failing to satisfy constraint qualifications and obtaining non-exact solutions of the NLP subproblems are both linked to the issue of cycling. For constraint qualifications, the link is that they are key to ensuring that separation is achieved. For non-exact solutions, the link is that no guarantee of separation can be given for points close to, but not equal to, the optimal solution.

To connect the issues of constraint qualifications and non-exact solutions, Theorem~\ref{thm:Slater} was formulated and proved. It shows that one can construct a point for the gradient cut which is feasible with respect to the nonlinear constraints but at which separation nonetheless fails. This means that, if the feasible set is small due to the nonlinear constraints, then high precision of the NLP solution is required for separation.

Finally, we argued that adding ECP cuts in conjunction with the OA algorithm provides a way of handling the issues described. We proved formally in Theorem~\ref{thm:algo finite} that adding ECP cuts at suboptimal integer assignments together with OA cuts at other points leads to a finite algorithm. We argue that this means that implementing a standard OA algorithm with ECP cuts as a fallback whenever integer assignments are repeated is a theoretically sound way of handling the issue of cycling.

\newpage
\bibliographystyle{amsplain}
\bibliography{references}

\providecommand{\bysame}{\leavevmode\hbox to3em{\hrulefill}\thinspace}
\providecommand{\MR}{\relax\ifhmode\unskip\space\fi MR }
\providecommand{\MRhref}[2]{%
  \href{http://www.ams.org/mathscinet-getitem?mr=#1}{#2}
}
\providecommand{\href}[2]{#2}
\begin{thebibliography}{10}

\bibitem{knitro}
Artelys, \emph{Knitro 15.1}, 2025, \url{https://www.artelys.com/app/docs/knitro/}.

\bibitem{belotti2013mixed}
P.~Belotti, C.~Kirches, S.~Leyffer, J.~Linderoth, J.~Luedtke, and A.~Mahajan, \emph{Mixed-integer nonlinear optimization}, Acta Numer. \textbf{22} (2013), 1--131, \url{https://doi.org/10.1017/S0962492913000032}.

\bibitem{bonami2008algorithmic}
P.~Bonami, L.~T. Biegler, A.~R. Conn, G.~Cornu{\'e}jols, I.~E. Grossmann, C.~D. Laird, J.~Lee, A.~Lodi, F.~Margot, N.~Sawaya, and A.~W{\"a}chter, \emph{An algorithmic framework for convex mixed integer nonlinear programs}, Discrete Optim. \textbf{5} (2008), no.~2, 186--204, \url{https://doi.org/10.1016/j.disopt.2006.10.011}.

\bibitem{minlplib}
M.~R. Bussieck, A.~S. Drud, and A.~Meeraus, \emph{{MINLPLib}---a collection of test models for mixed-integer nonlinear programming}, INFORMS J. Comput. \textbf{15} (2003), no.~1, 114--119, \url{https://doi.org/10.1287/ijoc.15.1.114.15159}.

\bibitem{dambrosio_2009}
C.~D'Ambrosio, J.~Lee, and A.~W{\"a}chter, \emph{A global-optimization algorithm for mixed-integer nonlinear programs having separable non-convexity}, Algorithms - {ESA} 2009 (Berlin, Heidelberg) (A.~Fiat and P.~Sanders, eds.), Lecture Notes in Computer Science, vol. 5757, Springer, 2009, \url{https://doi.org/10.1007/978-3-642-04128-0_10}, pp.~107--118.

\bibitem{delfino_2018}
A.~Delfino and W.~de~Oliveira, \emph{Outer-approximation algorithms for nonsmooth convex {MINLP} problems}, Optimization \textbf{67} (2018), no.~6, 797--819, \url{https://doi.org/10.1080/02331934.2018.1434173}.

\bibitem{duran_1986}
M.~A. Duran and I.~E. Grossmann, \emph{An outer-approximation algorithm for a class of mixed-integer nonlinear programs}, Math. Program. \textbf{36} (1986), no.~3, 307--339, \url{https://doi.org/10.1007/BF02592064}.

\bibitem{eronen_2014}
V.-P. Eronen, M.~M. M{\"a}kel{\"a}, and T.~Westerlund, \emph{On the generalization of {ECP} and {OA} methods to nonsmooth convex {MINLP} problems}, Optimization \textbf{63} (2014), no.~7, 1057--1073, \url{https://doi.org/10.1080/02331934.2012.712118}.

\bibitem{fletcher_1994}
R.~Fletcher and S.~Leyffer, \emph{Solving mixed integer nonlinear programs by outer approximation}, Math. Program. \textbf{66} (1994), 327--349, \url{https://doi.org/10.1007/BF01581153}.

\bibitem{dicopt}
I.~E. Grossmann, J.~Viswanathan, A.~Vecchietti, R.~Raman, and E.~Kalvelagen, \emph{{GAMS/DICOPT}: A discrete continuous optimization package}, 2003, \url{https://www.amsterdamoptimization.com/pdf/dicopt.pdf}.

\bibitem{aimms_aoa}
M.~Hunting, \emph{The {AIMMS} outer approximation algorithm for {MINLP}}, 2011, {AIMMS} technical report.

\bibitem{kronqvist2020using}
J.~Kronqvist, D.~E. Bernal, and I.~E. Grossmann, \emph{Using regularization and second order information in outer approximation for convex {MINLP}}, Math. Program. \textbf{180} (2020), no.~1, 285--310, \url{https://doi.org/10.1007/s10107-018-1356-3}.

\bibitem{kronqvist2019review}
J.~Kronqvist, D.~E. Bernal, A.~Lundell, and I.~E. Grossmann, \emph{A review and comparison of solvers for convex {MINLP}}, Optim. Eng. \textbf{20} (2019), no.~2, 397--455, \url{https://doi.org/10.1007/s11081-018-9411-8}.

\bibitem{kronqvist2025}
J.~Kronqvist, D.~E. Bernal~Neira, and I.~E. Grossmann, \emph{50 years of mixed-integer nonlinear and disjunctive programming}, Eur. J. Oper. Res. \textbf{331} (2025), no.~3, 687--705, \url{https://doi.org/10.1016/j.ejor.2025.07.016}.

\bibitem{kronqvist2016extended}
J.~Kronqvist, A.~Lundell, and T.~Westerlund, \emph{The extended supporting hyperplane algorithm for convex mixed-integer nonlinear programming}, J. Glob. Optim. \textbf{64} (2016), no.~2, 249--272, \url{https://doi.org/10.1007/s10898-015-0322-3}.

\bibitem{shot}
A.~Lundell, J.~Kronqvist, and T.~Westerlund, \emph{The supporting hyperplane optimization toolkit for convex {MINLP}}, J. Glob. Optim. \textbf{84} (2022), no.~1, 1--41, \url{https://doi.org/10.1007/s10898-022-01128-0}.

\bibitem{mahajan2021minotaur}
A.~Mahajan, S.~Leyffer, J.~Linderoth, J.~Luedtke, and T.~Munson, \emph{{Minotaur}: A mixed-integer nonlinear optimization toolkit}, Math. Program. Comput. \textbf{13} (2021), no.~2, 301--338, \url{https://doi.org/10.1007/s12532-020-00196-1}.

\bibitem{mangasarian_1969}
O.~L. Mangasarian, \emph{Nonlinear programming}, McGraw-Hill, New York, 1969.

\bibitem{peterson_1973}
D.~W. Peterson, \emph{A review of constraint qualifications in finite-dimensional spaces}, SIAM Rev. \textbf{15} (1973), no.~3, 639--654, \url{https://doi.org/10.1137/1015075}.

\bibitem{quesada1992lp}
I.~Quesada and I.~E. Grossmann, \emph{An {LP/NLP} based branch and bound algorithm for convex {MINLP} optimization problems}, Comput. Chem. Eng. \textbf{16} (1992), no.~10--11, 937--947, \url{https://doi.org/10.1016/0098-1354(92)80028-8}.

\bibitem{westerlund_1995}
T.~Westerlund and F.~Pettersson, \emph{An extended cutting plane method for solving convex {MINLP} problems}, Comput. Chem. Eng. \textbf{19} (1995), no.~Suppl. 1, S131--S136, Proceedings of {ESCAPE}-5, Bled, Slovenia, 11--14 June 1995; \url{https://doi.org/10.1016/0098-1354(95)87027-X}.

\end{thebibliography}

\begin{appendices}
\newpage
\section{Outer approximation algorithm}\label{app:OA}
Algorithm \ref{algo:OA} presents the outer approximation algorithm. It corresponds to Algorithm~1 of \cite{fletcher_1994}. Note that the formulation here is assumed to be \eqref{eq:General MINLP}, so that $f$ is linear and need not be linearized. Furthermore, the description in \cite{fletcher_1994} adds an upper bound of $\mathrm{UBD}$ to the objective function of $M$ and concludes optimality when $M$ is infeasible. Here, the stopping criterion used is $\mathrm{LBD} = \mathrm{UBD}$, which is equivalent.

\begin{algorithm}
\caption{Linear Outer Approximation}
\begin{algorithmic}[1]\label{algo:OA}

\STATE \textbf{Initialization:} Given $y^0$; set $i = 0$, $T^{-1} = \emptyset$, $S^{-1} = \emptyset$,
$\mathrm{LBD} = -\infty$, and $\mathrm{UBD} = \infty$.

\REPEAT

    \STATE Solve the subproblem $\mathrm{NLP}(y^i)$, or the feasibility problem $F(y^i)$ if
    $\mathrm{NLP}(y^i)$ is infeasible, and let the solution be $x^i$.

    \STATE Linearize the (active) constraint functions about $(x^i, y^i)$.
    Set
    \[
        T^i = T^{i-1} \cup \{i\} \quad \text{or} \quad S^i = S^{i-1} \cup \{i\}
    \]
    as appropriate.

    \IF{$\mathrm{NLP}(y^i)$ is feasible \AND $f^i < \mathrm{UBD}$}
        \STATE Update current best point: set $x^* = x^i$, $y^* = y^i$, $\mathrm{UBD} = f^i$.
    \ENDIF

    \STATE Solve the current relaxation $M^i$ of the master program $M$:
    \[
        \min_{x,\, y} \; f(x, y)
    \]
    subject to
    \begin{align*}
        &0 \geq g^j + [\nabla g^j]^\top \begin{pmatrix} x - x^j \\ y - y^j \end{pmatrix},
        \quad \forall\, j \in T^i, \\
        &0 \geq g^k + [\nabla g^k]^\top \begin{pmatrix} x - x^k \\ y - y^k \end{pmatrix},
        \quad \forall\, k \in S^i, \\
        &x \in \X, \quad y \in \Y \text{ integer},
    \end{align*}
    and set $\mathrm{LBD} = f(x^{i+1}, y^{i+1})$, giving a new integer assignment $y^{i+1}$.
    Set $i = i + 1$.

\UNTIL{$\mathrm{LBD} = \mathrm{UBD}$}

\end{algorithmic}
\end{algorithm}

\newpage
\section{Extended cutting plane algorithm}\label{app:ECP}
Algorithm \ref{algo:ECP} summarizes the extended cutting plane algorithm. It corresponds to the algorithm presented in \cite{westerlund_1995}. In this basic form, the algorithm is in general infinite. If a tolerance on feasibility with respect to the nonlinear constraints is added, however, it becomes finite.

\begin{algorithm}
\caption{Extended Cutting Plane (ECP) Method}
\begin{algorithmic}[1]\label{algo:ECP}

\STATE \textbf{Initialization:} Set $k = 0$ and $\Omega_0 = \Lc = \X \times \Y$.

\LOOP

    \STATE Solve the MILP subproblem over the current feasible set $\Omega_k$:
    \[
        \min_{x,\, y} \; \left\{ c^\top x + d^\top y \right\}
        \quad \text{s.t.} \quad (x, y) \in \Omega_k
    \]
    to obtain a new point $(x^k, y^k)$ with objective value
    $Z_k = c^\top x^k + d^\top y^k$.

    \STATE Identify the most violated constraint:
    \[
        f_k(x^k, y^k) = \max_i \left\{ g_i(x^k, y^k) \right\}
    \]

    \IF{$(x^k, y^k) \in \N$}
        \STATE \textbf{Converged.} The point $(x^k, y^k)$ is
        feasible and optimal. \textbf{Stop.}
    \ENDIF

    \STATE Generate a linearization of $f_k$ at $(x^k, y^k)$:
    \[
        l_k(x, y) = f_k(x^k, y^k)
        + \nabla f_k(x^k, y^k)^\top \begin{pmatrix} x - x^k \\ y - y^k \end{pmatrix}
    \]

    \STATE Update the feasible set by adding the new cutting plane constraint:
    \[
        \Omega_{k+1} = \Omega_k \cap \left\{ x, y \mid l_k(x, y) \leq 0 \right\}
    \]

    \STATE Set $k = k + 1$.

\ENDLOOP

\end{algorithmic}
\end{algorithm}

\end{appendices}

\end{document}